\documentclass[mathpazo]{cicp}

\usepackage{algorithm}
\usepackage{algorithmic}
\usepackage{bm}
\usepackage{xcolor}
\usepackage{float}
\usepackage{graphicx}
\usepackage{epstopdf}

\begin{document}
\title{Hermite spectral collocation methods for fractional PDEs in unbounded domains}


 \author[Tang et.~al.]{Tao Tang\affil{1}\comma\corrauth,
       Huifang Yuan\affil{2}, and Tao Zhou\affil{3}}
 \address{\affilnum{1}\ Department of Mathematics, Southern University of Sciences and Technology, Shenzhen, China. \\
           \affilnum{2}\ Department of Mathematics, Hong Kong Baptist University, Hong Kong, China.
           \affilnum{3}LSEC, Institute of Computational Mathematics, Academy of Mathematics and Systems Science,
Chinese Academy of Sciences, Beijing, China.}
 \emails{{\tt tangt@sustc.edu.cn} (T.~Tang), {\tt 13480510@life.hkbu.edu.hk} (H.~Yuan),
          {\tt tzhou@lsec.cc.ac.cn} (T.~Zhou)}

\begin{abstract}
This work is concerned with spectral collocation methods for fractional PDEs in unbounded domains. The method consists of expanding the solution with proper global basis functions and imposing collocation conditions on the Gauss-Hermite points. In this work, two Hermite-type functions are employed to serve as basis functions. Our main task is to find corresponding differentiation matrices which are computed recursively. Two important issues relevant to condition numbers and scaling factors will be discussed. Applications of the spectral collocation methods to multi-term fractional PDEs are also presented. Several numerical examples are carried out to demonstrate the effectiveness of the proposed methods.
\end{abstract}


\maketitle


\section{Introduction}
\label{sec1}
Many systems in science and engineering can be more accurately described by  using fractional partial differential equations (PDEs) rather than the traditional approaches \cite{physics_II,BFW98,dPQRV11}. This leads to an intensive investigation over the past two decades on efficient numerical methods for fractional PDEs. Among others, the finite difference method and the finite element method are two widely investigated methods in this direction, see, e.g., \cite{differenceHuang,differenceJi,elementJin,Jiao16,Ren15,TD15,WZ16} and references therein.

Another powerful approach for fractional models is the spectral methods. In this approach, the key is to construct suitable basis functions to handle to the solution singularities. Along this direction, a recent advance is brought by Karniadakis and co-authors who proposed the so-called Jacobi \textit{poly-fractonomials} based spectral methods \cite{spectralgeorge}. These bases are eigenfunctions of the corresponding fractional and tempered fractional Sturm-Liouville problems. Another approach that employs the generalized Jacobi functions (GJFs) is proposed by Shen et al. \cite{spectralshen}. Those bases are adapted to the fractional operator, as a fractional derivative of poly-fractonomials/GJFs is simply another poly-fractonomials/GJFs with different parameter. Consequently, fractional derivatives become a local operator in the physical space spanned by poly-fractonomials/GJFs, and this property leads to very efficient spectral methods for fractional PDEs in bounded domains. The poly-fractonomials/GJFs have been successfully applied to various fractional models \cite{Chen_Mao_Shen,KZK16,LZK16,SS}. However, compared to fractional PDEs in bounded domain, little works have been done for fractional PDEs defined on unbounded domains. Very recently, a spectral method for fraction differential equations in the half line is proposed in \cite{spectral Arab,LZK16} -- using the generalized Largurre functions as bases -- extending the idea of \cite{spectralgeorge}.

 When this paper is prepared, we noticed very recent work of Mao and Shen \cite{Mao_Shen} who proposed both the spectral Galerkin and collocation method for fractional PDEs in unbounded domains. However, the collocation method therein relies on an equivalent formulation in frequency space by the Fourier transform, and performs collocation methods to the equivalent formulation that involve forward/backward Hermite transform. In contrast, our collocation methods are \textit{direct} methods that based on the derivation of explicit DMs. Moreover, our approach can be easily applied to nonlinear problems as the differentiation matrices are constructed explicitly.

In this work, we aim at designing spectral collocation methods for fractional PDEs in unbounded domains (the whole space $\mathbb{R}^d$ ).
To better demonstrate our idea, we consider the following model equation:
\begin{equation}\label{Laplace}
\begin{cases}
  \left(-\Delta\right)^{\alpha/2}u(\mathbf{x}) + \gamma f(u) = g(\mathbf{x}), \quad &\mathbf{x}\in\mathbb{R}^d,\\
  u(\mathbf{x})=0,\quad &|\mathbf{x}| \to \infty,
  \end{cases}
\end{equation}
where $f(u)$ is a linear/nonlinear function of $u$, and the fractional Laplace operator is defined as \cite{Lan72}
\begin{equation}\label{sigular representation}
(-\Delta)^{\alpha/2}u(\mathbf{x})=C_{n,\alpha}\int_{\mathbb{R}^d}\dfrac{u(\mathbf{x})-u(\mathbf{y})}{|\mathbf{x-y}|^{n+\alpha}}d\mathbf{y}, \quad \textmd{with} \quad
C_{n,\alpha}=\dfrac{{\alpha}2^{\alpha-1}\Gamma\left(\dfrac{\alpha+n}{2}\right)}{\pi^{n/2}\Gamma\left(\dfrac{2-\alpha}{2}\right)}.
\end{equation}
Notice that the fractional Laplace operator $(-\Delta)^{\alpha/2}$, where $0<\alpha<2$, recovers the standard Laplace operator as $\alpha\to 2$.

For such problems that are defined on the whole space $\mathbb{R}^n$, there are alternative (equivalent) ways to define the fractional Laplace operator. For example, it can be defined as a pseudo-differential operator via the Fourier transform:
\begin{equation}\label{via fourier transform}
  \mathcal{F}{\left[(-\Delta)^{\alpha/2}u\right]}\left(\xi\right)=|\xi|^{\alpha}\mathcal{F}{\left[u\right]}\left(\xi\right).
\end{equation}

Our spectral collocation methods handling the above equation consist of two parts: expanding the solution with the basis functions and imposing collocation conditions on the Guass-Hermite points. In particular, we shall consider two types of expansion bases, that is, the normalized Hermite functions $\{e^{-x^2/2}H_n(x)\}_n$ and the over-scaled bases $\{e^{-x^2}H_n(x)\}_n.$ For both approaches, we shall derive explicit formulas for the associated differential matrix (DM), for which the components can be computed efficiently by using a recurrence formula. To deal with solutions with different decay rate, a scaling factor will be included in the expansion. Application to multi-term fractional Laplace equations will also be discussed. It is noticed that both methods admit spectral convergence for solutions with exponential decay in infinity.

The rest of the this paper is organized as follows. The next section provides some preliminaries for some special functions. Our spectral collocations methods are presented and discussed in Sections 3 and 4, for the over-scaled bases $\{e^{-x^2}H_n(x)\}_n$  and the normalized Hermite functions, respectively. In Section 5 we shall discuss an equivalent collocation scheme by using the Lagrange type bases. Numerical examples are presented in Section 6 to demonstrate the effectiveness of the proposed spectral methods. We finally give some concluding remarks in Section 7.

\section{Preliminaries}
\label{sec2}

This section will provide some preliminaries useful for designing our spectral collocation methods. For ease of notations, in this section we shall focus our attention to the one-dimensional case.

\subsection{Confluent hypergeometric functions}
We first introduce the definition of confluent hypergeometric function of the first kind which is defined by the following power series \cite{tableofintegrals}
\begin{equation}\label{expansion definition}
{}_1F_1\left(a,b;x\right)=\sum_{k=0}^{\infty}\dfrac{\left(a\right)_{k}}{\left(b\right)_{k}}\dfrac{x^k}{k!},
\end{equation}
where $\left(a\right)_{k}$ is the Pochhammer symbol defined as
\begin{equation*}
\left(a\right)_{0}=1, \quad \left(a\right)_{k}=a\left(a+1\right)\left(a+2\right)...\left(a+k-1\right).
\end{equation*}
By the above definitions, we can easily get
\begin{equation*}
\dfrac{d^{k}}{dx^{k}}{}_1F_1\left(a,b;x\right)=\dfrac{\left(a\right)_{k}}{\left(b\right)_{k}}{}_1F_{1}\left(a+k,b+k;x\right).
\end{equation*}
We also have the following integral representation \cite{recurrenceof1F1}
\begin{equation}\label{integral definition}
{}_1F_1\left(a,b;x\right)=\dfrac{\Gamma\left(b\right)}{\Gamma\left(b-a\right)\Gamma\left(a\right)}
\int_{0}^{1}e^{tx}t^{a-1}\left(1-t\right)^{b-a-1}dt, \quad a, \, b >0.
\end{equation}
It is noticed that the confluent hypergeometric function satisfies the Kummer's transformation formula
\begin{equation}\label{Kummer property}
{}_1F_1\left(a,b;-x\right)=e^{-x}{}_1F_{1}\left(b-a,b;x\right).
\end{equation}
It it easy to check that the following recurrence formula holds
\begin{equation}\label{recurrence relation}
\left(2a-b+x\right){}_1F_1\left(a,b;x\right)=a{}_1F_1\left(a+1,b;x\right)-\left(b-a\right){}_1F_1\left(a-1,b;x\right).
\end{equation}

\subsection{Hermite polynomials/functions}
The Hermite polynomials, denoted by $H_{n}\left(x\right)$, $n \geqslant 0$, $x\in\mathbb{R}$, are defined by the following three-term recurrence relation (see e.g., \cite{SWT,Wang16}):
\begin{align*}
H_{0}(x)=1, \quad H_{1}(x)=2x, \quad
H_{n+1}(x)=2xH_{n}(x)-2nH_{n-1}(x), \quad n\geq 1.
\end{align*}
The Hermite polynomials are orthogonal with respect to the weight function $\omega\left(x\right)=e^{-x^2},$ namely,
\begin{equation*}
\int_{\mathbb{R}}H_{m}(x)H_{n}(x)e^{-x^2}dx=\gamma_{n}\delta_{mn},  \quad \gamma_{n}=\sqrt{\pi}2^{n}n!.
\end{equation*}
It is well known that Hermite polynomials and the confluent hypergeometric function satisfy the following formulas:
\begin{align}\label{hermite and confluent}
  &H_{2n}(x)=(-1)^n\dfrac{(2n)!}{n!}{}_1F_1\left(-n, 1/2;x^2\right);\\
  &H_{2n+1}(x)=(-1)^n\dfrac{(2n+1)!}{n!}2x{}_1F_1\left(-n, 3/2;x^2\right).
\end{align}
The corresponding normalized Hermite functions are defined as
\begin{equation}\label{normalized Hermite function}
\widehat{H}_{n}(x)=\dfrac{1}{\sqrt{2^{n}n!}}e^{-x^2/2}H_{n}(x).
\end{equation}
and they are orthogonal with respect to the weight function $\omega(x)=1$, i.e.,
\begin{equation*}
\int_{\mathbb{R}}\widehat{H}_{m}(x)\widehat{H}_{n}(x)dx=\sqrt{\pi}\delta_{mn}.
\end{equation*}
We shall discuss the spectral collocation methods based on the above Hermite functions in Section 4.

We shall also discuss the spectral collocation methods based on the \textit{over-scaled} bases that defined as follows
\begin{equation}\label{associated Hermite function}
\widetilde{H}_{n}(x)= e^{-x^2/2}\widehat{H}_{n}(x).
\end{equation}
It is easy to see that the generalized Hermite defined here is orthogonal with respect to the weight function $\omega(x)=e^{x^2}$, i.e.,
\begin{equation*}
\int_{\mathbb{R}}\widetilde{H}_{m}(x)\widetilde{H}_{n}(x)e^{x^2}dx=\sqrt{\pi}\delta_{mn}.
\end{equation*}
Such bases were first proposed by Brinkman in \cite{Brinkman} and have been well studied
in physics, see e.g., \cite{Risken}.
The above over-scaled basis was first proposed by Brinkman when studying the so-called Fokker-Planck equations,
where the velocity part of the probability distribution function was expanded in Hermite functions
(\ref{associated Hermite function}).
His approach has become one of the most popular methods used for solving the Fokker-Planck equation, see, e.g., \cite{Fok,Risken}.

\subsection{Bessel functions}
We shall also use properties of generalized Bessel functions.  Recall that the Bessel function of order $\mu$ is defined as
\begin{equation}
J_{\mu}(x)=\sum_{m=0}^\infty\, \frac{(-1)^m}{m! \,\Gamma(m+\mu+1)} \left(\frac{x}{2}\right)^{2m+\mu}.
\end{equation}
In particular we have
\begin{equation}\label{bassel_half}
J_{-\frac{1}{2}}(x)=\sqrt{\frac{2}{\pi x}}\cos{x}.
\end{equation}
For the Bessel functions it holds in \cite{tableofintegrals} that
\begin{equation}\label{bassel_half_II}
\int_{\mathbb{R}^+}J_{\mu}(bt)\exp(-p^2 t^2)t^{\nu-1}dt = \frac{ \left(\frac{b}{2p}\right)^{\mu} \Gamma(\frac{\mu+\nu}{2})}{2p^{\nu}\Gamma(\mu+1)}  {}_1F_1\left(\frac{\mu+\nu}{2}, \mu+1, -\frac{b^2}{4p^2}\right).
\end{equation}
It is easy to check that for Bessel function with integer parameter it holds
\begin{equation}\label{bassel_p1}
\frac{d}{dx}\left\{x^{-n}J_{n}(x)\right\}=-x^{-n}J_{n+1}(x).
\end{equation}
Let $m, n$ be two integers, and let $\nu \geq -n-1, \,\, \mu \geq -m-1$ be two real numbers,  we denote
\begin{equation}
\mu + \nu + m + n := \delta; \quad  \mu + \nu - m - n := \zeta.
\end{equation}
Then for $a> 0$ it holds (see e.g, \cite{integrals and series}, P.216)
\begin{align}\label{bassel_p2}
&\quad \int_{0}^{a}x^{v+2n+1}\left(a^2-x^2\right)^{m+\mu/2}J_{\mu}\left(b\sqrt{a^2-x^2}\right)J_{v}(cx)dx \nonumber\\
&=a^{\zeta+1}b^{\mu}c^{v}\left(\tfrac{\partial}{b\partial b}\right)^{m}\left(\tfrac{\partial}{c\partial c}\right)^{n}\left[\left(b^2+c^2\right)^{-\tfrac{\delta+1}{2}}J_{\delta+1}\left(a\sqrt{b^2+c^2}\right)\right].
\end{align}

\section{Spectral collocation methods based on $\{\widetilde{H}_{n}(x)\}_n$}
\label{sec3}
In this section, we first consider the spectral collocation method based on the $\{\widetilde{H}_{n}(x)\}_n.$
We assume that the solution admits an exponential decay in infinity, and we shall approximate $u(x)$ with a finite sum of the the basis $\{\widetilde{H}_{n}(x)\}_n,$ i.e.,
\begin{equation}\label{expansion}
  u(x)\approx u_{N}(x)=\sum_{n=0}^{N-1}c_{n}\widetilde{H}_{n}(x).
\end{equation}
By inserting the above expansion into the fractional PDE (\ref{Laplace}), we obtain
\begin{equation}\label{fractional expansion}
  \left(-\Delta\right)^{\alpha/2}u_{N}(x) + \gamma f(u_N(x)) = g(x).
\end{equation}
Let $\{x_{i}\}_{i=0}^{N-1}$ be the roots of the $N$-th order Hermite polynomials, we then impose the collocation conditions on these collocation points, which yields
\begin{align*}
\sum_{n=0}^{N-1}c_{n}(-\Delta)^{\alpha/2}\widetilde{H}_{n}(x_{i}) + \gamma f(u_N(x_i)) =g(x_{i}), \quad  i=0, 1,..., N-1.
\end{align*}
Then we can write the above equations into the following system
\begin{equation*}
  \widetilde{\mathcal{D}}^{\alpha} \mathbf{c} + \gamma F(\mathbf{c})=\mathbf{g},
\end{equation*}
where $\mathbf{c}=\left(c_0, ..., c_{N-1}\right)^{T}$ is the unknown coefficient vector, and $\widetilde{\mathcal{D}}^{\alpha} \in \mathbb{R}^{N\times N}$ is the differential matrix with components
\begin{equation*}
  \widetilde{\mathcal{D}}_{i,j}^{\alpha}=(-\Delta)^{\alpha/2}\widetilde{H}_j(x_{i}), \quad i , j = 0, 1,..., N-1.
\end{equation*}
By solving the above linear system one gets an approximated solution $u_{N}(x)$. Next, we shall derive explicit formulas for the components of the differential matrix.

\subsection{The one dimensional case}
We first consider the even terms $\widetilde{H}_{2n}.$ Precisely, we have the following theorem
\begin{theorem}
For $0< \alpha <2 ,$ we have
\begin{equation*}
  (-\Delta)^{\alpha/2}\widetilde{H}_{2n}\left(x\right)=2^{\alpha}\dfrac{(-1)^{n}\sqrt{(2n)!}}{2^{n}n!}\dfrac{\Gamma\left(n+\frac{\alpha}{2}+\frac{1}{2}\right)}{\Gamma\left(n+\frac{1}{2}\right)}{}_1F_{1}\left(n+\dfrac{\alpha}{2}+\dfrac{1}{2},\dfrac{1}{2},-x^2\right).
\end{equation*}
\end{theorem}
\begin{proof}
Consider the forward Fourier transform, we have
\begin{align*}
\mathcal{F}{\left[\widetilde{H}_{2n}\right]}(\xi)
&=\frac{1}{\sqrt{2\pi}\sqrt{2^{2n}(2n)!}}\int_{\mathbb{R}}\exp(-x^2)H_{2n}(x)e^{-\textmd{i}x\xi}dx\\
&=\frac{2}{\sqrt{2\pi}\sqrt{2^{2n}(2n)!}}\int_{\mathbb{R}^+}\exp(-x^2)H_{2n}(x)\cos(x\xi)dx\\
&=\frac{(-1)^{n}}{\sqrt{2}\sqrt{2^{2n}(2n)!}}\xi^{2n}e^{-\frac{\xi^{2}}{4}}.
\end{align*}
Then by considering the inverse Fourier transform we have
\begin{align*}
(-\Delta)^{\alpha/2}\widetilde{H}_{2n}(x)
&=\frac{1}{\sqrt{2\pi}}\int_{\mathbb{R}}|\xi|^{\alpha}\mathcal{F}{\left[\widetilde{H}_{2n}\right]}(\xi)e^{\textmd{i}x\xi}d\xi\\
&=\frac{1}{\sqrt{2\pi}}\frac{(-1)^{n}}{\sqrt{2}\sqrt{2^{2n}(2n)!}}\int_{\mathbb{R}}|\xi|^{\alpha}\xi^{2n}e^{-\frac{\xi^{2}}{4}}e^{\textmd{i}x\xi}d\xi\\
&=\frac{2}{2\sqrt{\pi}}\frac{(-1)^{n}}{\sqrt{2^{2n}(2n)!}}\int_{\mathbb{R}^+}\xi^{2n+\alpha}e^{-\frac{\xi^{2}}{4}}\cos(x\xi)d\xi\\
&=\frac{(-1)^{n}}{\sqrt{\pi}\sqrt{2^{2n}(2n)!}}2^{2n+\alpha}\Gamma\left(\frac{2n+\alpha+1}{2}\right){}_1F_1\left(\frac{2n+\alpha+1}{2},\frac{1}{2};-x^2\right)\\
&=2^{\alpha}\dfrac{(-1)^{n}\sqrt{(2n)!}}{2^{n}n!}\dfrac{\Gamma\left(n+\frac{\alpha}{2}+\frac{1}{2}\right)}{\Gamma\left(n+\frac{1}{2}\right)}{}_1F_{1}\left(n+\dfrac{\alpha}{2}+\dfrac{1}{2},\dfrac{1}{2},-x^2\right).
\end{align*}
This completes the proof.
\end{proof}
\vskip .3cm

In the above derivation, we have adopted some integral formulas in \cite{integral transform}. Using similar arguments, for the odd terms $\{\widetilde{H}_{2n+1}(x)\},$ we have the following theorem
\begin{theorem}
For $0<\alpha<2,$ it holds
\begin{eqnarray}
  && (-\Delta)^{\alpha/2}\widetilde{H}_{2n+1}(x) \nonumber \\
  &=& 2^{\alpha+1}\dfrac{(-1)^n\sqrt{(2n+1)!}}
  {2^{n+\frac{1}{2}}n!}\dfrac{\Gamma\left(n+\frac{\alpha}{2}+\frac{3}{2}\right)}
  {\Gamma\left(n+\frac{3}{2}\right)}x{}_1F_{1}\left(n+\dfrac{\alpha}{2}+\dfrac{3}{2},\dfrac{3}{2},-x^2\right).
\end{eqnarray}
\end{theorem}

By Theorems 3.1-3.2, we get the following explicit formula for the components of the differential matrix (DM)
\begin{eqnarray}\label{DM}
&& \widetilde{\mathcal{D}}_{ij}^{\alpha}=(-\Delta)^{\alpha/2}\widetilde{H}_{j}(x_{i}) \\
&=&\begin{cases}
2^{\alpha}\dfrac{(-1)^{n}\sqrt{(2n)!}}{2^{n}n!}\dfrac{\Gamma\left(n+\frac{\alpha}{2}+\frac{1}{2}\right)}{\Gamma\left(n+\frac{1}{2}\right)}{}_1F_{1}\left(n+\dfrac{\alpha}{2}+\dfrac{1}{2},\dfrac{1}{2},-x_{i}^2\right), \quad j=2n; \nonumber\\
2^{\alpha+1}\dfrac{(-1)^n\sqrt{(2n+1)!}}{2^{n+\frac{1}{2}}n!}\dfrac{\Gamma\left(n+\frac{\alpha}{2}+\frac{3}{2}\right)}{
\Gamma\left(n+\frac{3}{2}\right)}x_{i}{}_1F_{1}\left(n+\dfrac{\alpha}{2}+\dfrac{3}{2},\dfrac{3}{2},-x_{i}^2\right),\quad j=2n+1. \nonumber\\
\end{cases}
\end{eqnarray}
We now summarize the procedure for computing the components of the DM:
\begin{itemize}
\item For each $x_i,$ compute the quantities $\widetilde{\mathcal{D}}^\alpha_{ij}$ with $j=0,1,2,3$, by the above formula. Notice that one has to deal with confluent hypergeometric function ${}_1F_1$ and this may be non-trivial, as by definition (\ref{expansion definition}), this is an infinite expansion.  Nevertheless, one can find a fast \& accurate algorithm for example in \cite{computent1F1}.

\item Compute the quantities $\widetilde{\mathcal{D}}^\alpha_{ij}$ for $4<j \leq N-1$ in a recurrence way using the recurrence formula (\ref{recurrence relation}).
\end{itemize}
In general, the above DM is easy to construct, and in fact, the matrix can be stored in priori (offline). We provide in Fig. \ref{tildeHcondition} the condition number of the differential matrix with respect to the number of collocation points $N.$ It is noticed that the condition number grows very fast with respect to $N.$ Thus, efficient pre-conditioners should be designed for practice applications.
\begin{figure}[htbp]
\centering
\includegraphics[width=8cm]{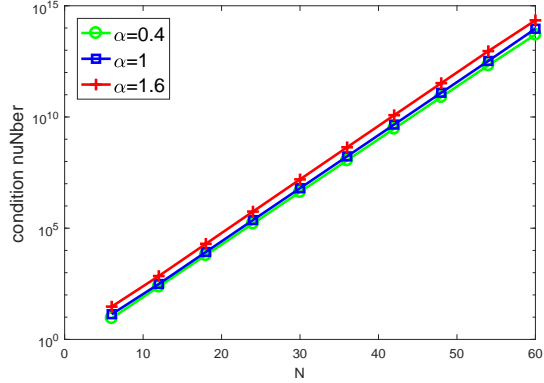}
\caption{Condition number of the differentiation matrix} \label{tildeHcondition}
\end{figure}
Notice that we have used the definition (\ref{via fourier transform}) to derive the differential matrix, however, one can also adopt the definition (\ref{sigular representation}) to derive a similar differential matrix.

\subsection{The two dimensional case}

For multi-dimensional cases, the bases used will be the tensorized 1D bases. Here we take the two dimensional case as an example. To this end, the two dimensional bases take the following  form $\{\widetilde{H}_{n}(x)\widetilde{H}_{m}(y)\}_{n,m}.$

By \eqref{via fourier transform}, we know that
\begin{align*}
\mathcal{F}\left[\widetilde{H}_{n}\widetilde{H}_{m}\right](\xi,\eta)
&=\frac{1}{2\pi}\int_{\mathbb{R}}\int_{\mathbb{R}}\widetilde{H}_{n}(x)\widetilde{H}_{m}(y)e^{-\textmd{i}x\xi}e^{-\textmd{i}y\eta}dxdy \nonumber \\
&=\frac{1}{2\pi}\int_{\mathbb{R}}\widetilde{H}_{n}(x)e^{-\textmd{i}x\xi}dx\int_{\mathbb{R}}\widetilde{H}_{m}(y)e^{-\textmd{i}y\eta}dy=\mathcal{F}{\left[\widetilde{H}_{n}\right]}(\xi)\mathcal{F}{\left[\widetilde{H}_{m}\right]}(\eta).
\end{align*}
Then, using the inverse Fourier transform gives
\begin{align}
&\qquad(-\Delta)^{\alpha/2}\left[\widetilde{H}_{n}(x)\widetilde{H}_{m}(y)\right] \nonumber \\
&=\frac{1}{2\pi}\int_{\mathbb{R}}\int_{\mathbb{R}}\left(\xi^2+\eta^2\right)^{\frac{\alpha}{2}}
\mathcal{F}{\left[\widetilde{H}_{n}\right]}(\xi)\mathcal{F}{\left[\widetilde{H}_{m}\right]}(\eta)e^{\textmd{i}x\xi}e^{\textmd{i}y\eta}d\xi d\eta.
\end{align}
We take the double-even term $\{\widetilde{H}_{2n}(x)\widetilde{H}_{2m}(y)\}$ as an example. We denote
\begin{equation*}
L_{nm}=\frac{(-1)^{n}}{\sqrt{2}\sqrt{2^{2n}(2n)!}}\frac{(-1)^{m}}{\sqrt{2}\sqrt{2^{2m}(2m)!}}.
\end{equation*}
Then we have
\begin{align*}
&\qquad(-\Delta)^{\alpha/2}\left[\widetilde{H}_{2n}(x)\widetilde{H}_{2m}(y)\right] \\
&=\frac{L_{nm}}{2\pi}\int_{\mathbb{R}}\int_{\mathbb{R}}\left(\xi^2+\eta^2\right)^{\frac{\alpha}
{2}}\xi^{2n}e^{-\frac{\xi^{2}}{4}}\eta^{2m}e^{-\frac{\eta^{2}}{4}}e^{\textmd{i}x\xi}e^{\textmd{i}y\eta}d\xi d\eta\\
&=\frac{4L_{nm}}{2\pi}\int_{\mathbb{R}^+}\int_{\mathbb{R}^+}\left(\xi^2+\eta^2\right)^{\frac{\alpha}{2}}\xi^{2n}e^{-\frac{\xi^{2}}{4}}\eta^{2m}e^{-\frac{\eta^{2}}{4}}\cos(x\xi)\cos(y\eta)d\xi d\eta\\
&=\frac{4L_{nm}}{2\pi} \,\int_{\mathbb{R}^+}\mathcal{I}(x,y,\rho)\,\,\rho^{2n+2m+\alpha+1}e^{-\frac{\rho^2}{4}}d\rho,
\end{align*}
where
$$
\mathcal{I}(x,y,\rho)=\int_{0}^{\frac{\pi}{2}}(\cos{\theta})^{2n}(\sin{\theta})^{2m}
\cos(x\rho\cos{\theta})\cos(y\rho\sin{\theta})d\theta.
$$
By the definition of the Bessel function (\ref{bassel_half}) and the property (\ref{bassel_p2})  we have
\begin{align*}
&\quad\int_{0}^{\frac{\pi}{2}}(\cos{\theta})^{2n}(\sin{\theta})^{2m}\cos(p\cos{\theta})\cos(q\sin{\theta})d\theta\\
&=\frac{\pi}{2}(pq)^{\frac{1}{2}}\int_{0}^{1}(1-x^2)^{n-\frac{1}{4}}x^{2m+\frac{1}{2}}J_{-\frac{1}{2}}(p\sqrt{1-x^2})J_{-\frac{1}{2}}(qx)dx\\
&=\frac{\pi}{2}\left(\tfrac{\partial}{p\partial p}\right)^{n}\left(\tfrac{\partial}{q\partial q}\right)^{m}\left[\left(\sqrt{p^2+q^2}\right)^{-(n+m)}J_{n+m}\left(\sqrt{p^2+q^2}\right)\right].
\end{align*}
Notice that by property (\ref{bassel_p1}) we have
\begin{align*}
&\quad \left(\tfrac{\partial}{p\partial p}\right)^{n}\left(\tfrac{\partial}{q\partial q}\right)^{m}\left[\left(\sqrt{p^2+q^2}\right)^{-(n+m)}J_{n+m}\left(\sqrt{p^2+q^2}\right)\right]\\
&=(-1)^{n+m}\left(\sqrt{p^2+q^2}\right)^{-(2n+2m)}J_{2n+2m}\left(\sqrt{p^2+q^2}\right).
\end{align*}
Now, by replacing $p$ and $q$ with $\rho x$ and $\rho y$, we obtain
\begin{align*}
\mathcal{I}(x,y,\rho)=\frac{\pi}{2}(-1)^{n+m}\rho^{-(2n+2m)}\left(\sqrt{x^2+y^2}\right)^{-(2n+2m)}J_{2n+2m}\left(\rho\sqrt{x^2+y^2}\right).
\end{align*}
We then do inverse Fourier transform to obtain (for $t=m+n$)
\begin{small}
\begin{align*}
&\quad \left(-\Delta\right)^{\alpha/2}\left[\widetilde{H}_{2n}(x)\widetilde{H}_{2m}(y)\right]\\
\vspace{0.05cm}
&=(-1)^{t}L_{nm}(x^2+y^2)^{-t}\int_{\mathbb{R}^+}\rho^{\alpha+1}e^{-\frac{\rho^2}{4}}J_{2t}\left(\rho\sqrt{x^2+y^2}\right)d\rho\\
&=\frac{2^{\alpha}\Gamma(t+\frac{\alpha}{2}+1)}{\sqrt{2^{2n}(2n)!}\sqrt{2^{2m}(2m)!}\Gamma(2t+1)}
{}_1F_1\left(t+\frac{\alpha}{2}+1;2t+1;-\left(x^2+y^2\right)\right).
\end{align*}
\end{small}
In the above derivation, we have used property (\ref{bassel_half_II}). In a similar way, we can derive similar results for the bases of forms  $\widetilde{H}_{2n+1}(x)\widetilde{H}_{2m}(y)$, $\widetilde{H}_{2n}(x)\widetilde{H}_{2m+1}(y)$ and $\widetilde{H}_{2n+1}(x)\widetilde{H}_{2m+1}(y)$. To summaries these results in an unified way, we introduce the parameter $\delta_1$ and $\delta_2$, with $\delta_1,\delta_2 \in \{0,1\}$. And let
\begin{align*}
&a=n+m+\frac{\alpha}{2}+\delta_1+\delta_2+1; \quad  b=2n+2m+\delta_1+\delta_2+1.
\end{align*}
Using similar arguments as above, we can derive that
\begin{align*}
(-\Delta)^{\alpha/2}\left[\widetilde{H}_{2n+\delta_1}(x)\widetilde{H}_{2m+\delta_2}(y)\right]= C_{x,y}(\alpha, m,n,\delta_1,\delta_2) \frac{\Gamma(a)}{\Gamma(b)} {}_1F_1\left(a;b;-\left(x^2+y^2\right)\right),
\end{align*}
with
\begin{align*}
C_{x,y}(\alpha,m,n,\delta_1,\delta_2)=\frac{2^{\alpha+\delta_1+\delta_2}x^{\delta_1}y^{\delta_2}}{\sqrt{2^{2n+\delta_1}(2n+\delta_1)!}
\sqrt{2^{2m+\delta_2}(2m+\delta_2)!}}.
\end{align*}

We finally provide the explicit formulas for the components of the DM as following
\begin{align}
\widetilde{\mathcal{D}}&_{(i-1)*N+j,(p-1)*N+q}^{\alpha}=(-\Delta)^{\alpha/2}\left\{\widetilde{H}_{p}(x_i)\widetilde{H}_{q}(y_j)\right\}\\
&=\begin{cases}
\dfrac{2^{\alpha}\Gamma(n+m+\frac{\alpha}{2}+1){}_1F_1\left(n+m+\frac{\alpha}{2}+1;2n+2m+1;-\left(x_i^2+y_j^2\right)\right)}{\sqrt{2^{2n}(2n)!}\sqrt{2^{2m}(2m)!}\Gamma(2n+2m+1)}, \nonumber \\
\qquad \qquad \qquad \qquad \qquad \qquad \qquad \qquad \qquad \qquad \qquad    p=2n, q=2m;  \nonumber \\
\dfrac{2^{\alpha+1}\Gamma(n+m+\frac{\alpha}{2}+2)y_{j}{}_1F_1\left(n+m+\frac{\alpha}{2}+2;2n+2m+2;-\left(x_i^2+y_j^2\right)\right)}{\sqrt{2^{2n}(2n)!}\sqrt{2^{2m+1}(2m+1)!}\Gamma(2n+2m+2)}, \nonumber \\
\qquad \qquad \qquad \qquad \qquad \qquad \qquad \qquad \qquad  \qquad \qquad   p=2n, q=2m+1;  \nonumber \\
\dfrac{2^{\alpha+1}\Gamma(n+m+\frac{\alpha}{2}+2)x_{i}{}_1F_1\left(n+m+\frac{\alpha}{2}+2;2n+2m+2;-\left(x_i^2+y_j^2\right)\right)}{\sqrt{2^{2n+1}(2n+1)!}\sqrt{2^{2m}(2m)!}\Gamma(2n+2m+2)}, \nonumber \\
\qquad \qquad \qquad \qquad \qquad \qquad \qquad \qquad \qquad  \qquad \qquad   p=2n+1, q=2m;  \nonumber \\
\dfrac{2^{\alpha+2}\Gamma(n+m+\frac{\alpha}{2}+3)x_{i}y_{j}{}_1F_1\left(n+m+\frac{\alpha}{2}+3;2n+2m+3;-\left(x_i^2+y_j^2\right)\right)}{\sqrt{2^{2n+1}(2n+1)!}\sqrt{2^{2m+1}(2m+1)!}\Gamma(2n+2m+3)}, \nonumber \\
\qquad \qquad \qquad \qquad \qquad \qquad \qquad \qquad \qquad \qquad \qquad    p=2n+1, q=2m+1.
\end{cases}
\end{align}
Again, the above components can be computed in a similar procedure as in the one-dimensional case. We remark that one may also derive such formulas for the three dimensional case, and we omit it here as it may involves complex notations.

\subsection{The use of scaling factors}
It is well known that for Hermite-type spectral methods,  the convergence rate deteriorates if the decay rates between the solution and the bases function have a relatively large gap.
A remedy to fix this problem is to use the so-called
scaling factor  \cite{Sun,Scaling}. We now introduce the basic idea of the scaling factor and choose the one-dimensional case as an example for illustration.
To this end, let
$u(x)$ be a function that decay exponentially, namely,
\begin{equation}\label{eq:decay}
|u(x)|\sim 0,  \quad \forall \;  |x|>M,
\end{equation}
where $M>0$ is some constant.
The idea of using the scaling factor is to expand $u$ as
\begin{equation}\label{eq:herfun_p=6}
u(x)=\sum_{n=0}^{N-1} c_n \widetilde{H}_n(r x) \,\,\Leftrightarrow\,\,
u\left(x/r\right)=\sum_{n=0}^{N-1} c_n \widetilde{H}_n(x),
\end{equation}
where $r >0$ is a scaling factor. The key point of using
$r$ is to scale the Hermite-Gauss nodes $\{x_k\}_{k=0}^{N-1}$ so that the collocation points $\{x_k/r\}_{k=0}^{N-1}$
are well within the effective support of $u.$ This suggests the following choice
\begin{equation}
\max_{0 \leq k \leq N-1}\{|x_k|\}/r \leq M \quad  \Rightarrow \quad r = \max_{0\leq k \leq N-1}\{|x_k|\}/M.
\end{equation}
We remark however, in practice, finding the quantity $M$ may be non-trivial and thus the optimal scaling factor is hard to obtain in general.

We now show how to include the scaling factor in our spectral collocation methods for the fractional PDEs. Let us illustrate the idea in 1D. We now seek to the following expansion
\begin{equation*}
  u_{N}(x)=\sum_{n=0}^{N-1}c_{n}\widetilde{H}_{n}(rx).
\end{equation*}
By inserting the expansion to the fractional PDE (\ref{Laplace}) and imposing the collocation condition we obtain
\begin{equation}
\widetilde{\mathcal{D}}^{\alpha,r} \mathbf{c} + \gamma F^r(\mathbf{c})=\mathbf{g},
\end{equation}
where the $\widetilde{\mathcal{D}}^{\alpha,r}$ is the differential matrix with components
\begin{align}
\widetilde{\mathcal{D}}_{ij}^{\alpha,r} =(-\Delta)^{\alpha/2}\widetilde{H}_{j}( rx_{i}), \quad i, j =0, ... , N-1.
\end{align}
We need to deal with the fractional Laplacian of $\widetilde{H}_{j}\left(rx\right)$. To this end, suppose that
\begin{equation*}
\left(-\Delta\right)^{\alpha/2}v\left(x\right)=\phi\left(x\right),
\end{equation*}
Then the fractional Laplacian of $v_{r}\left(x\right)=v\left(rx\right)$ is
\begin{equation*}
\left(-\Delta\right)^{\alpha/2}v_{r}\left(x\right)=r^{\alpha}\phi\left(rx\right).
\end{equation*}
This can be simply proved by using the definition \eqref{sigular representation}. By using together the above argument and the DM formula (\ref{DM}), we have
\begin{align}
\widetilde{\mathcal{D}}_{ij}^{\alpha,r}&=(-\Delta)^{\alpha/2}\widetilde{H}_{j}( rx_{i}) \\
&=\begin{cases}
(2r)^{\alpha}\dfrac{(-1)^{n}\sqrt{(2n)!}}{2^{n}n!}\dfrac{\Gamma\left(n+\frac{\alpha}{2}+\frac{1}{2}\right)}
{\Gamma\left(n+\frac{1}{2}\right)}{}_1F_{1}\left(n+\frac{\alpha}{2}+\frac{1}{2},\frac{1}{2},-z_{i}^2\right), \quad j =2n ; \nonumber\\
2^{\alpha+1}r^{\alpha}\dfrac{\!(-1)^n\!\!\sqrt{(2n+1)!}\!}{2^{n+\tfrac{1}{2}}n!}\dfrac
{\Gamma\left(n+\tfrac{\alpha}{2}+\tfrac{3}{2}\right)}{\Gamma\left(n+\tfrac{3}{2}\right)}z_{i}{}_1F_{1}\left(\!n+\tfrac{\alpha}{2}+\frac{3}{2}\!,\tfrac{3}{2},-z_{i}^2\right), \quad j =2n+1. \nonumber\\
\end{cases}
\end{align}
where $z_i = r x_i$ for $i=0,1, ..., N-1.$

\subsection{Applications to multi-term fractional PDEs}
In this section, we claim that our spectral collocation method can be easily applied to the multi-term fractional PDEs.  We shall still take the one-dimensional case as an example. Consider the following multi-term fractional PDEs
\begin{equation}\label{discrete multi-term}
\sum_{j=1}^{J} (-\Delta)^{\alpha_{j}/2}u(x) + \gamma f(u) = g(x), \quad x\in\mathbb{R}.
\end{equation}
The above problem is motivated by the approximation of distributed order fractional models
using a quadrature rule, see e.g. \cite{distributeDiethelm,distributed_I,NMTMA}.

For the above multi-term models, by inserting the Hermite expansion (\ref{expansion}) into the equation and
imposing the collocation condition, one gets the following system:
\begin{equation}
\widetilde{\mathcal{D}}^{\mathcal{J}} \mathbf{c} + \gamma F(\mathbf{c})=\mathbf{g} \qquad  \textmd{with} \qquad  \widetilde{\mathcal{D}}^{\mathcal{J}}= \sum_{j=1}^{J} \,\widetilde{\mathcal{D}}^{\alpha_{j}}.
\end{equation}
Notice that the components of each differential matrix $\widetilde{\mathcal{D}}^{\alpha_{j}}$ can be computed by the explicit formula $(\ref{DM}).$

\section{Spectral collocation methods based on the normalized Hermite functions $\{\widetilde{H}_n\}_n.$}
\label{sec4}
In the last section, we have proposed the spectral methods with the over-scaled bases $\{\widetilde{H}_n\}_n.$ While the associated DM is easy to compute, its condition number grows fast with respect to the number of collocation points, and this is due to the poor property of the bases. In this section, we shall discuss the spectral collocation methods based on the normalized Hermite functions $\{\widehat{H}_n\}_n.$ The main task is still to derive the explicit formula for the differential matrix.

\subsection{The one dimensional case}
The aim is the present the explicit formula for $\widehat{\mathcal{D}}_{mn}^{\alpha}=(-\Delta)^{\alpha/2}\widehat{H}_{n}(x_{m}).$ This also relies on the forward/inverse Fourier transform as done in the last section.

By the forward Fourier transform, we Notice that
\begin{equation}
\mathcal{F}\left\{\widehat{H}_{n}(x)\right\}=\frac{1}{\sqrt{2\pi}}\int_{\mathbb{R}}\widehat{H}_{n}(x)e^{-\textmd{i}\xi x}dx=(-\textmd{i})^n\widehat{H}_{n}(\xi).
\end{equation}
That is, the normalized Hermite functions $\{\widehat{H}_{n}\}_n$ are eigenfunctions of the Fourier transform with eigenvalues $\{(-\textmd{i})^n\}_n$ with $\textmd{i}=\sqrt{-1}.$

A more difficult part is to compute the inverse Fourier transform. By the above fact and (\ref{via fourier transform}), we have to deal with the inverse fourier transform of  $\left\{(-\textmd{i})^n |\xi|^{\alpha} \widehat{H}_{n}(\xi)\right\}_n.$ Notice that we can write
\begin{equation}{\label{expansion_II}}
  \widehat{H}_{n}(\xi)=\sum_{k=0}^{n}\hat{a}_{n,k}\exp(-\xi^2/2)\xi^{k},   \quad n=0, 1, ..., N-1
\end{equation}
with $\widehat{a}_{n,k}=\frac{1}{\sqrt{2^{n}n!}}a_{n,k}$ for $k\leq n,$ where $a_{n,k}$ can be computed in a recursion way
\begin{align*}
& a_{0,0}=1, \,\, a_{1,0}=0, \,\, a_{1,1}=2;\\
& a_{n+1,k}=-a_{n,k+1}, \quad k=0;\\
& a_{n+1,k}=2a_{n,k-1}-(k+1)a_{n,k+1}, \quad k>0.
\end{align*}
Let us first consider the even terms with $k=2m,$ the inverse Fourier transform yields
 \begin{align*}
 & \mathcal{F}^{-1}{\left[\exp(-\xi^2/2)\xi^{2m}|\xi|^{\alpha}\right]}(x)\\
 &=\frac{1}{\sqrt{2\pi}}\int_{\mathbb{R}}\exp(-\xi^2/2)\xi^{2m}|\xi|^{\alpha}e^{\textmd{i}\xi x}d\xi\\
 &=\frac{2}{\sqrt{2\pi}}\int_{\mathbb{R}^+}\exp(-\xi^2/2)\xi^{2m+\alpha}\cos(\xi x)d\xi\\
 &=\frac{2^{\frac{2m+\alpha}{2}}}{\sqrt{\pi}}\Gamma\left(\frac{2m+1+\alpha}{2}\right){}_{1}F_{1}\left(\frac{2m+1+\alpha}{2},\frac{1}{2},-\frac{x^2}{2}\right).
 \end{align*}
Then for the odd terms with $k=2m+1$, it holds
 \begin{align*}
 & \mathcal{F}^{-1}{\left[\exp(-\xi^2/2)\xi^{2m+1}|\xi|^{\alpha}\right]}(x) \\
   &=\frac{1}{\sqrt{2\pi}}\int_{\mathbb{R}}\exp(-\xi^2/2)\xi^{2m+1}|\xi|^{\alpha}e^{\textmd{i}\xi x}d\xi\\
   &=\frac{2\textmd{i}}{\sqrt{2\pi}}\int_{\mathbb{R}^+}\exp(-\xi^2/2)\xi^{2m+1+\alpha}\sin(\xi x)d\xi\\
   &=\frac{2^{\frac{2m+2+\alpha}{2}}\textmd{i}}{\sqrt{\pi}}\Gamma\left(\frac{2m+3+\alpha}{2}\right)x{}_{1}F_{1}\left(\frac{2m+3+\alpha}{2},\frac{3}{2},-\frac{x^2}{2}\right).
 \end{align*}
For ease of notations, we denote
\begin{equation*}
F_{k}(x)=\mathcal{F}^{-1}{\left[\exp(-\xi^2/2)\xi^{k}|\xi|^{\alpha}\right]}(x), \quad k=0, 1, ..., N-1.
\end{equation*}
Then by (\ref{expansion_II}) the components of the differentiation matrix yield
\begin{equation}\label{normal_1D}
\widehat{\mathcal{D}}_{mn}^{\alpha}=(-\Delta)^{\alpha/2}\widehat{H}_{n}(x_{m})=(-\textmd{i})^n
\sum\limits_{k=0}^{N-1}\hat{a}_{n,k}F_{k}\left(x_m\right), \quad 0 \leq n, m \leq N-1.
\end{equation}
Notice that when $k>n,$ we set $a_{n,k}=0.$ In Fig. \ref{hatHcondition} we present the condition number of this differential matrix with respect to $N.$
\begin{figure}[htbp]
\centering
\includegraphics[width=8cm]{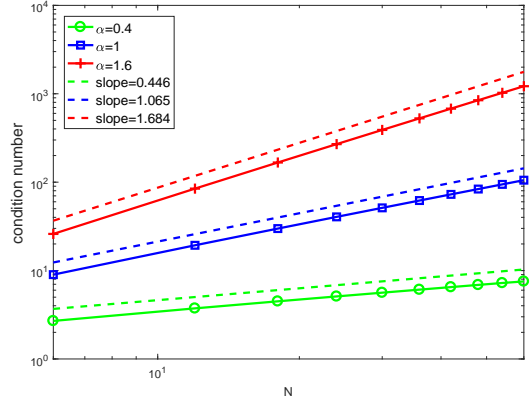}
\caption{Condition number of the differentiation matrix} \label{hatHcondition}
\end{figure}
It is noticed from Fig. \ref{hatHcondition} that the condition number grows algebraically with $N$ -- much well behaved than the previous case (where the over scaled bases are used). In fact the computation complexity of the DM above is almost the same as the formula (\ref{DM}) since each components here is simply a evaluation of a linear combination of the confluent hypergeometric functions.

\subsection{The two dimensional case}
We now consider the two dimensional case. By the forward Fourier transform we have
\begin{align}
\mathcal{F}\left[\widehat{H}_{n}\widehat{H}_{m}\right](\xi,\eta)
&=\frac{1}{2\pi}\int_{\mathbb{R}}\int_{\mathbb{R}}
\widehat{H}_{n}(x)\widehat{H}_{m}(y)e^{-\textmd{i}x\xi}e^{-\textmd{i}y\eta}dxdy \nonumber \\
&=\frac{1}{2\pi}\int_{\mathbb{R}}\widehat{H}_{n}(x)e^{-\textmd{i}x\xi}dx
\int_{\mathbb{R}}\widehat{H}_{m}(y)e^{-\textmd{i}y\eta}dy \nonumber \\
&=(-\textmd{i})^{n+m}\widehat{H}_{n}(\xi)\widehat{H}_{m}(\eta).
\end{align}
Then, by the inverse Fourier transform we obtain
\begin{align*}
(-\Delta)^{\alpha/2}&\left[\widehat{H}_{n}(x)\widehat{H}_{m}(y)\right]=\frac{(-\textmd{i})^{n+m}}{(2\pi)}\int_{\mathbb{R}}
\int_{\mathbb{R}}\left(\xi^2+\eta^2\right)^{\frac{\alpha}{2}}\widehat{H}_{n}(\xi)\widehat{H}_{m}(\eta)e^{\textmd{i}x\xi}e^{\textmd{i}y\eta}d\xi d\eta.
\end{align*}
Similar as in the one dimensional case, we expand $H_{n}(\xi)H_{m}(\eta)$ as a combination of $\xi^{k}\eta^{l}.$ Inspired by (\ref{expansion_II}) we have
\begin{equation*}
\widehat{H}_{n}(\xi)\widehat{H}_{m}(\eta)=\sum_{k=0}^{N-1}\sum_{l=0}^{N-1}\hat{a}_{n,k}\hat{a}_{m,l}
\exp\left(-(\xi^2+\eta^2)/2\right)\xi^{k}\eta^{l}.
\end{equation*}
Next, we should deal with the inverse Fourier transform of terms like  $$\left(\xi^2+\eta^2\right)^{\frac{\alpha}{2}}\exp\left(-\left(\xi^2+\eta^2\right)/2\right)\xi^{k}\eta^{l}, \quad 0\leq k,l \leq N-1. $$
The derivation is similar as in the one dimensional case, and thus we omit the details here. To summarize, we introduce the matrix $F(\xi,\eta)$ as
\begin{align}
&F_{k,l}=\mathcal{F}^{-1}\left\{\left(\xi^2+\eta^2\right)^{\frac{\alpha}{2}}\exp\big(-(\xi^2+\eta^2)/2\big)\xi^{k}\eta^{l}\right\}\nonumber \\
=&\begin{cases}
\dfrac{2^{\alpha/2}(-1)^{p+q}\Gamma(p+q+\frac{\alpha}{2}+1){}_1F_1\left(p+q+\frac{\alpha}{2}+1;2p+2q+1;-\frac{x^2+y^2}{2}\right)}{2^{p+q}\Gamma(2p+2q+1)}, \nonumber \\
\qquad \qquad \qquad \qquad \qquad \qquad \qquad \qquad \qquad \qquad \qquad    k=2p, \,\,l=2q;  \nonumber \\
\dfrac{2^{(\alpha+1)/2}(-1)^{p+q}i\Gamma(p+q+\frac{\alpha}{2}+2)y{}_1F_1\left(p+q+\frac{\alpha}{2}+2;2p+2q+2;-\frac{x^2+y^2}{2}\right)}{2^{p+q}\Gamma(2p+2q+2)}, \nonumber \\
\qquad \qquad \qquad \qquad \qquad \qquad \qquad \qquad \qquad  \qquad \qquad   k=2p,\,\,l=2q+1;  \nonumber \\
\dfrac{2^{(\alpha+1)/2}(-1)^{p+q}i\Gamma(p+q+\frac{\alpha}{2}+2)x{}_1F_1\left(p+q+\frac{\alpha}{2}+2;2p+2q+2;-\frac{x^2+y^2}{2}\right)}{2^{p+q}\Gamma(2p+2q+2)}, \nonumber \\
\qquad \qquad \qquad \qquad \qquad \qquad \qquad \qquad \qquad  \qquad \qquad   k=2p+1, \,\,l=2q;  \nonumber \\
\dfrac{-2^{\alpha/2+1}(-1)^{p+q}\Gamma(p+q+\frac{\alpha}{2}+3)xy{}_1F_1\left(p+q+\frac{\alpha}{2}+3;2p+2q+3;-\frac{x^2+y^2}{2}\right)}{2^{p+q}\Gamma(2p+2q+3)}, \nonumber \\
\qquad \qquad \qquad \qquad \qquad \qquad \qquad \qquad \qquad \qquad \qquad    k=2p+1, \,\,l=2q+1.
\end{cases}
\end{align}
Finally, we provide the explicit formula for the components of the differentiation matrix as follows
\begin{align}
\widehat{\mathcal{D}}_{(p-1)*N+q,(n-1)*N+m}^{\alpha}
&=(-\Delta)^{\alpha/2}\left\{\widehat{H}_{n}(x_p)\widehat{H}_{m}(y_q)\right\} \nonumber \\
&=(-i)^{n+m}\sum\limits_{k=0}^{N-1}\sum\limits_{l=0}^{N-1}\hat{a}_{n,k}\hat{a}_{m,l}F_{k,l}(x_p,y_q).
\end{align}
\begin{remark}
Similar as in Section 3, for the spectral collocation methods with the normalized Hermite functions, one can easily include the scaling factors with slight modifications to the differential matrix. And furthermore, the application to multi-term fractional PDEs is also straightforward.
\end{remark}

\section{Spectral collocation methods with Lagrange type bases}
\label{sec5}
In the above sections, we have derived explicit formulas for the differential matrices by using the Hermite-type bases. In this section, we shall discuss the equivalent spectral collocation methods with Lagrange type bases. For notation simplicity, we shall only present the one dimensional case. Notice that for this type of bases, we still search an (equivalent) approximated solution in the same finite space as in the above sections, however, such an approach indeed results in different differential matrices. By using the Lagrange type bases, we approximate the solution in the following way
\begin{equation*}
u_{N}(x)=\sum_{j=0}^{N-1}u_{j}h_{j}(x), \quad \textmd{with} \quad h_{j}(x_k) = \delta_{jk}, \quad 0\leq j,k\leq N-1.
\end{equation*}
Here the Lagrange type bases $\{h_{j}(x)\}_{j=1}^N$ are defined as
\begin{equation*}
h_{j}(x)=\frac{e^{-x^2/2}}{e^{-x_{j}^2/2}}\prod_{i=0,i\neq j}^{N-1}\frac{x-x_{j}}{x_{i}-x_{j}}, \quad 0\leq j\leq N-1.
\end{equation*}
The associated points $\left\{x_{j}\right\}_{j=0}^{N-1}$ are the Gauss-Hermite points. It is clear that we can express each Lagrange-type basis with the normalized Hermite functions, i.e.,
\begin{equation*}
h_j(x)=\sum_{k=0}^{N-1} \, b_{k}^{j}\widehat{H}_{k}\left(x\right), \quad \textmd{with} \quad  b_{k}^{j}=\frac{1}{\sqrt{\pi}}\widehat{H}_{k}\left(x_{j}\right)\hat{\omega}_{j},  \quad 0\leq j,k\leq N-1,
\end{equation*}
where $\{\hat{\omega}_{j}\}_{j=0}^{N-1}$ is the weights of the Gauss quadrature rule associated with the Hermite functions, which are defined as:
$$\hat{\omega}_{j}=\frac{\sqrt{\pi}}{N\widehat{H}_{N-1}^{2}\left(x_{j}\right)},\quad j=0,1,...,N-1.$$
Consequently,  we can easily derive the associated differential matrix $\widehat{D}^{L,\alpha}$ with Lagrange type bases
\begin{equation*}
\widehat{D}_{i,j}^{L,\alpha}=(-\Delta)^{\alpha/2}h_{j}(x_i)=\sum_{k=0}^{N-1}b_{k}^{j}(-\Delta)^{\alpha/2}\widehat{H}_{k}\left(x_i\right).
\end{equation*}
The quantities $(-\Delta)^{\alpha/2}\widehat{H}_{k}\left(x_i\right)$ can be obtained via equation (\ref{normal_1D}). The condition number of the above differential matrix with respect to the order $N$ is presented in Fig. \ref{collocationcond}. Again, we can see that the condition number grows algebraically respect to the order $N$.
\begin{figure}[htbp]
\centering
\includegraphics[width=8cm]{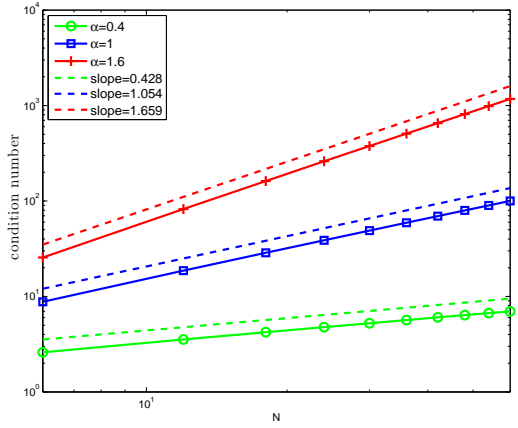}
\caption{Condition number of the differentiation matrix with Lagrange-type bases}\label{collocationcond}
\end{figure}

\section{Numerical examples}
\label{sec6}
In this section, we shall present several constructive examples to show the convergence property of our spectral collocation methods.  In all our computations, we shall report the numerical error both in the weighted norm $e_{w}$ and in the maximum norm $e_{m},$ which are defined respectively as
\begin{align*}
e_{w}=\|u(x)-u_{N}(x)\|_{\omega}^{2}, \quad e_{m}=\max_{j}\left|u(x_{j})-u_{N}(x_{j})\right|.
\end{align*}
Here $\omega(x)=e^{x^2}$ for the bases $\{\widetilde{H}_n(x)\}_n$ and $\omega(x)=1$ for the normalized Hermite functions.

\subsection{The fractional Laplace equation}
Our first example is the fractional Laplace equation
\begin{equation}\label{Example 5.1}
\begin{cases}
  \left(-\Delta\right)^{\alpha/2}u(x)  = g(x), \quad &x\in\mathbb{R},\\
  u(x)=0,\quad &|x| \to \infty.
  \end{cases}
\end{equation}
The right hand side is chosen such that the exact solution is $u(x)=\exp(-x^2)\textmd{sin}x$. Notice that for a given exact solution $u(x),$ the right hand side does not necessary have explicit formula, and in such cases, we shall compute it by expanding $u(x)$ with a large enough number of the basis functions. This is also true for all other examples.
\begin{figure}[htbp]
\centering
\includegraphics[width=12cm,height=5cm]{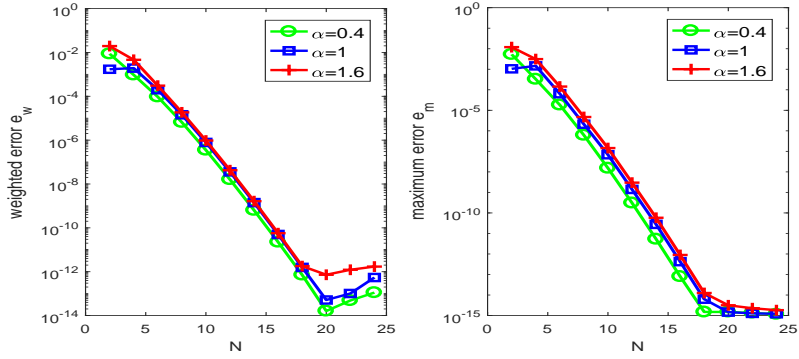}
\caption{Numerical error with the over-scaled bases for (\ref{Example 5.1}) with $u(x)=\exp(-x^2)\sin{x}.$ Left: weighted norm. Right: maximum norm.}\label{tildeHerror1}
\end{figure}

We perform computations with different fractional order, i.e., $\alpha=0.4$, $1$ and $1.6.$ The numerical errors against the numerical of points $N$ with the over scaled bases and the normalized Hermite functions are presented in Fig. \ref{tildeHerror1} and Fig. \ref{hatHerror1}, respectively. For the over-scaled bases, no scaling factor is needed, i.e, $r=1,$ while for the normalized Hermite functions, we choose $r=\sqrt{2}$. In Fig. \ref{tildeHerror1} and Fig. \ref{hatHerror1}, the numerical error in weighted norm and in maximum norm are reported in the left plot and right plot, respectively. It is clear that spectral convergence is obtained for all cases of $\alpha.$ While no scaling factor is needed for the over scaled bases, it is noticed that the convergence is polluted when a larger number of collocation points are used, due to the fast grow of the condition number.

\begin{figure}[htbp]
  \centering
  \includegraphics[width=12cm,height=5cm]{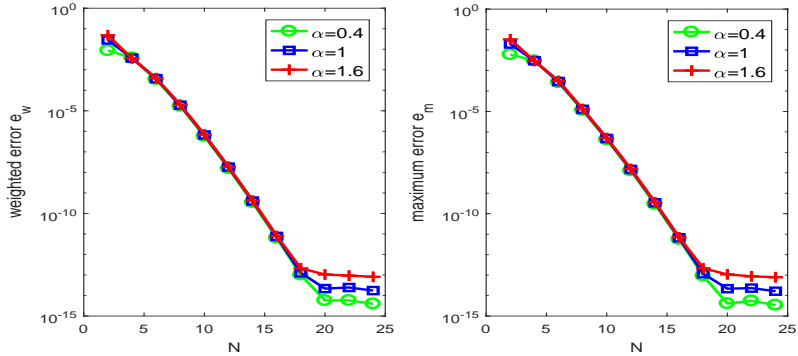}
  \caption{Numerical error with the normalized Hermite functions for (\ref{Example 5.1}) with $u(x)=\exp(-x^2)\sin{x}.$ Left: weighted norm. Right: maximum norm.}\label{hatHerror1}
\end{figure}

\subsection{A linear fractional PDE}

Next, we consider the following fractional PDE
\begin{equation}\label{Example 5.2}
\begin{cases}
(-\Delta)^{\alpha/2}u\left(x\right)+2u\left(x\right)=f\left(x\right),\quad &x\in\mathbb{R}\\
u\left(x\right)=0,\quad &x \to \infty
\end{cases}
\end{equation}
We first test the performance of the over-scaled bases. We set $u(x)=\exp(-\frac{x^2}{2})x^2\cos(x),$ and the right hand side can be computed accordingly.  Again, we consider $\alpha=0.4,1$ and $1.6$. The numerical results with a scaling factor $r=1/\sqrt{2}$ and without the scaling factor ($r=1$) are reported in Fig. \ref{example5.2.1} and Fig. \ref{example121}, respectively. It is clear seen that using a proper scaling factor results in faster convergence  rate for both the weighted error and the maximum error.
\begin{figure}[htbp]
\centering
\includegraphics[width=12cm,height=5cm]{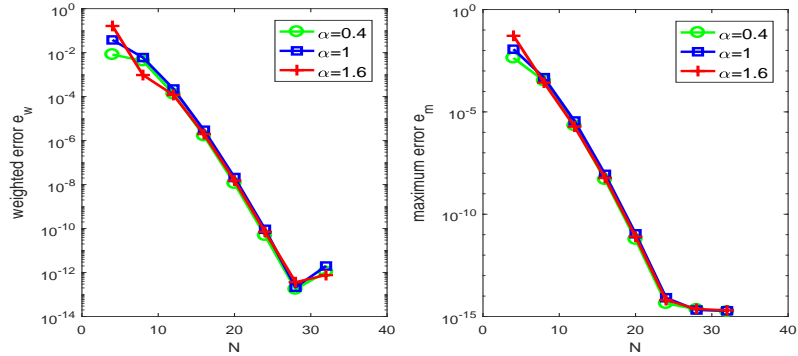}
\caption{Numerical error with the over scaled bases for (\ref{Example 5.2}) with exact solution $u(x)=\exp(-\frac{x^2}{2})x^2\cos(x).$ The scaling factor is $r=1/\sqrt{2}$. Left: weighted norm. Right: maximum norm.}\label{example5.2.1}
\end{figure}

\begin{figure}[htbp]
\centering
\includegraphics[width=12cm,height=5cm]{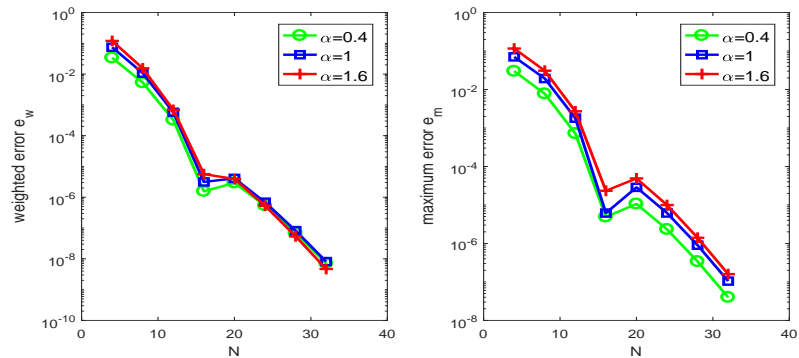}
\caption{Numerical error with the over scaled bases for (\ref{Example 5.2}) with exact solution $u(x)=\exp(-\frac{x^2}{2})x^2\cos(x).$ The scaling factor is $r=1$. Left: weighted norm. Right: maximum norm.}\label{example121}
\end{figure}

Now we test the normalized Hermite functions. For the same equation we choose the right hand side such that the solution yields $u(x)=\exp(-2x^2)x^2\cos(x).$ It is clear that the optimal scaling factor is $r=2$. The numerical results with a scaling factor $r=2$ and without a scaling factor (i.e., $r=1$) are reported Fig. \ref{hatHerror2} and Fig. \ref{hatHerror21}, respectively. Again, it is shown that a proper scaling factor can be useful to speed up the convergence.
\begin{figure}[htbp]
\centering
\includegraphics[width=12cm,height=5cm]{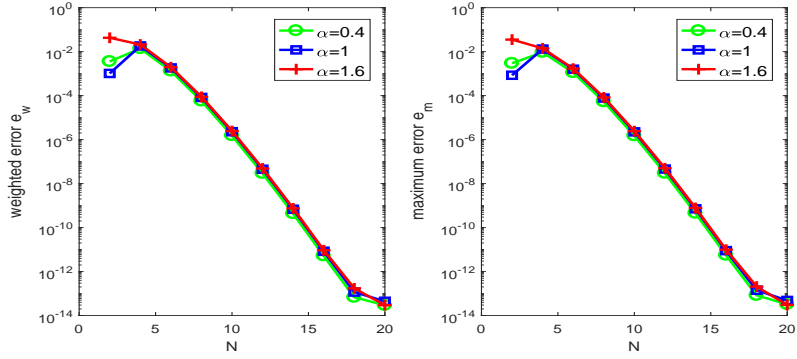}
\caption{Numerical error with the normalized Hermite functions for (\ref{Example 5.2}) with exact solution $u(x)=\exp(-2x^2)x^2\cos(x).$ The scaling factor is chosen as $r=2$. Left: weighted norm. Right: maximum norm.}\label{hatHerror2}
\end{figure}

\subsection{A two-dimensional example}
We now consider a two dimensional example, and the equation considered is
\begin{equation}\label{Example 5.3}
(-\Delta)^{\alpha/2}u(x,y)+2 u(x,y) = g(x,y).
\end{equation}
The exact solution is chosen as $u(x,y)=\exp(-(x^2+y^2))\sin(x+y).$ We also perform the computations with $\alpha=0.4, \,1$ and $1.6.$ To avoid too much pictures, here we only test the performance of the over-scaled bases. The numerical errors in weighted and maximum norm against the numerical of number of collocation points $N$ are presented in Fig. \ref{example5.3}. Spectral convergence is again observed.
\begin{figure}[htbp]
\centering
\includegraphics[width=12cm,height=5cm]{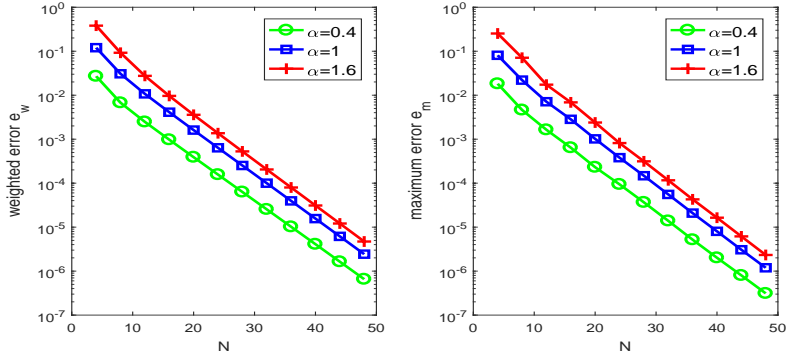}
\caption{Numerical error with normalized Hermite functions for (\ref{Example 5.2}) with exact solution $u(x)=\exp(-2x^2)x^2\cos(x).$ The scaling factor is chosen as $r=1$. Left: weighted norm. Right: maximum norm.}\label{hatHerror21}
\end{figure}

\begin{figure}[htbp]
\centering
\includegraphics[width=12cm,height=5cm]{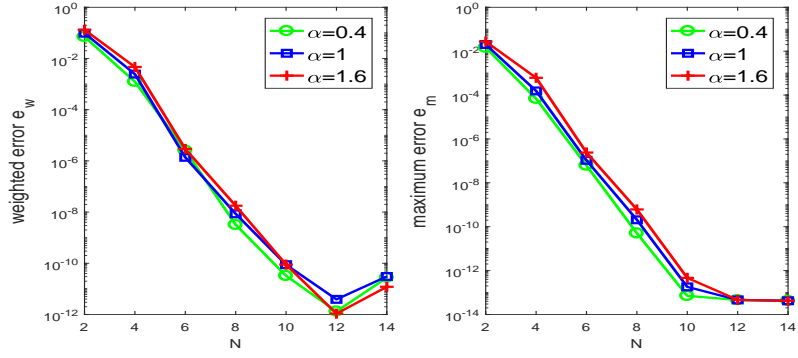}
\caption{A two dimensional example (\ref{Example 5.3}) with the exact solution $u(x,y)=\exp(-(x^2+y^2))\sin(x+y)$. Left: numerical error in weighted norm. Right: numerical error in maximum norm.}\label{example5.3}
\end{figure}


\subsection{A multi-term fractional model}
Our next example the multi-term Laplacian equation:
\begin{equation}\label{Example 5.4}
\sum_{j=1}^{J} (-\Delta)^{\alpha_{j}/2}u(x) = g(x), \quad x\in\mathbb{R}.
\end{equation}
Here we set $J=4$ and $\{\alpha_j\}_{j=1}^J$ are chosen as the transformed Legendre-Gauss points:
\begin{equation}
\alpha_1= 0.139,  \quad \alpha_2= 0.660, \quad \alpha_3= 1.340, \quad \alpha_4= 1.861.
\end{equation}
We set the exact solution to be $u(x)=\textmd{exp}(- 3x^2/2)\left(\sin{x}+x^6+x^2\cos{x}\right)$ and the right hand side can be computed accordingly. Again, we test the performance of the over-scaled bases. In this example, we consider scaling factors $r=\sqrt{1.5}$, $\sqrt{1.3}$, and the approach without a scaling factor, i.e., $r=1$. The corresponding numerical results are reported in Fig. \ref{example5.5}.
\begin{figure}[htbp]
\centering
\includegraphics[width=12cm,height=5cm]{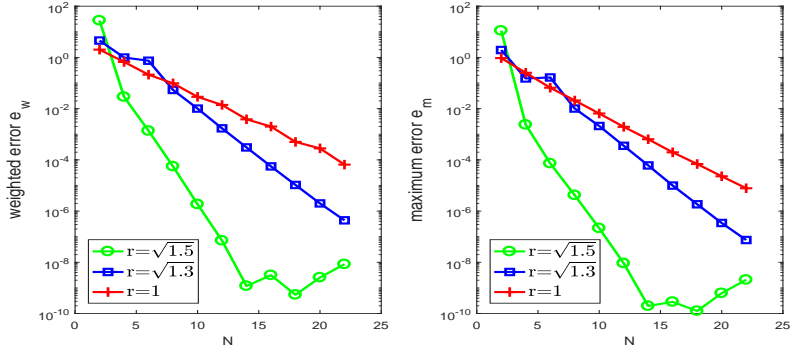}
\caption{The multi-term fractional Laplace equation with exact solution $u(x)=e^{-3x^2/2}(\sin{x}+x^6+x^2\cos{x}).$ Left: numerical error in weighted norm. Right: numerical error in maximum norm. }\label{example5.5}
\end{figure}
We can see that both the weighted error and maximum error decay fast for the case of $r=\sqrt{1.5}$. And this indicates the effectiveness of using a scaling in improving convergence rate.

%

\subsection{A nonlinear example}
Our next example is a nonlinear fractional PDE
\begin{equation}\label{Example 5.5}
(-\Delta)^{\alpha/2}u(x)+ u^2(x)=g(x)\\
\end{equation}
For the over scaled bases $\widetilde{H}_{n}(x)$, we set the exact solution as $u(x)=\exp(-x^2)(\sin(x)+x^2)$.
For the normalized Hermite functions $\widehat{H}_{n}(x)$, the exact solution is chose to be $u(x)=\exp(-x^2/2)(\sin(x)+x^2).$ In our computations, for each expansion number $N,$ we use the Newton iteration method with a tolerance $10^{-16}$ to deal with the nonlinear term. The performance of collocation method with the normalized Hermite functions are presented Fig. \ref{hatHerroraddcu2}, and it is shown that the method yields a spectral convergence rate.
%

\subsection{An eigenvalue problem}
\begin{figure}[htbp]
\centering
\includegraphics[width=12cm,height=5cm]{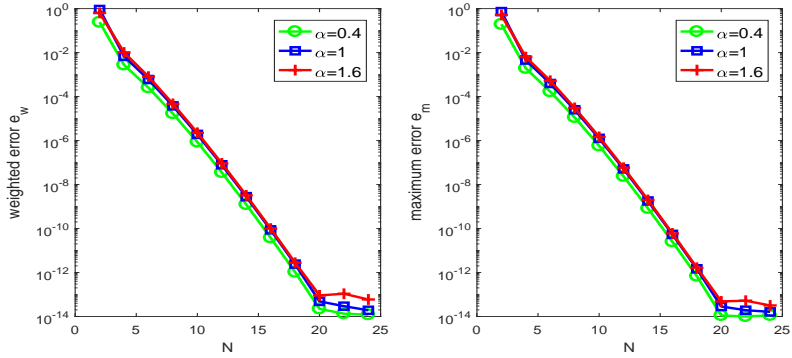}
\caption{Numerical error with $\hat{H}_n$ for the nonlinear problem (\ref{Example 5.5}).}
\label{hatHerroraddcu2}
\end{figure}
Finally we consider the following eigenvalues problem
\begin{equation}\label{Example 5.6}
\left((-\Delta)^{\alpha/2}+x^2\right)u(x)=\lambda u(x)
\end{equation}
The above eigenvalue problem with $\alpha=1$ has been analyzed in \cite{eigenvalue}. In particular, the eigenvalues of this problems is given by
\begin{align*}
\lambda_{2k-1}=-a_{k}', \quad \lambda_{2k}=-a_{k}, \quad k=1,2,...,
\end{align*}
where $a_{k}$ and $a_{k}'$ are the roots of the following Airy function and its derivative (in the decreasing order)
\begin{equation*}
A(x)=\frac{1}{\pi}\int_{0}^{\infty}\cos\left(\frac{t^3}{3}+xt\right)dt.
\end{equation*}

In this example, we shall compute the first three eigenvalues by the spectral collocation method. The exact eigenvalues are
\begin{align*}
\lambda_{1}\approx1.01879297164747, \quad \lambda_{2}\approx2.33810741045976, \quad \lambda_{3}\approx3.24819758217983.
\end{align*}
Numerical result are presented in Fig. \ref{eigenvalue} with $\log\log$ scale. An algebraic decay is observed and this is due to the algebraic decay (non-exponential decay) of eigenvalues.

\begin{figure}[htbp]
\centering
\includegraphics[width=8cm]{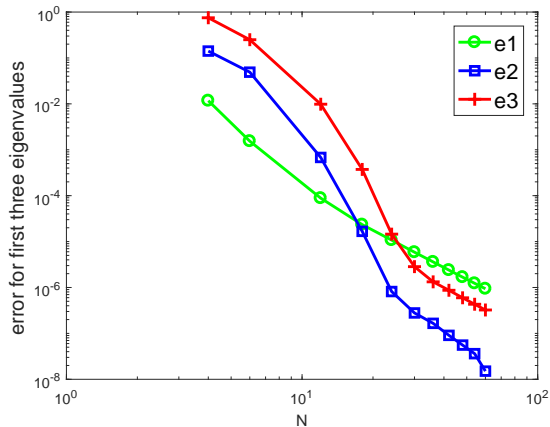}
\caption{Numerical errors for the first three eigenvalues.}
\label{eigenvalue}
\end{figure}

\section{Summary and Conclusion}
\label{sec7}
In this work, we have proposed a spectral collocation method based on Hermite functions for fractional PDEs in unbounded domain. One basis function is associated with the over-scaled weight $\exp(-x^2)H_n(x)$ and another is with the standard normalized-weight $\exp(-x^2/2)H_n(x)$, each has some advantages or popularity in practice. For both approaches, explicit differentiation matrices are derived. To deal with solutions with fast or  slow decay rate, a scaling factor in spectral implementation is discussed.

Although the numerical experiments indicate the spectral rate of convergence, there are still several issues requiring future investigations:

\begin{itemize}
\item Rigorous convergence analysis is not covered in this work, which is still an ongoing work.

\item For over-scaled Hermite functions (which was used extensively in physics, see \cite{Risken}) the condition numbers of the associated matrices grow fast with respect to $N$. Thus, it is useful to investigate some efficient pre-conditioners in this case.

\item We only provide some ad-hoc discussions on the scaling factors. It will be more meaningful to provide some more practical guidance on the optimal choice of the scaling factors, see, e.g., \cite{Sun}.
\end{itemize}

\section*{Acknowledgments}
This work is partially supported by the National Natural Science Foundations of China under grant numbers 91630312, 91630203, 11571351, and 11731006. The second author is supported by a Hong Kong PhD Fellowship. The last author is supported by the science challenge project (No. TZ2016001), NCMIS, and the youth innovation promotion association (CAS).

\end{document}